\documentclass{article}
\usepackage{amssymb,amsrefs,amsthm,amsfonts,amsmath,latexsym}
\usepackage{color}
\usepackage[latin1]{inputenc}
\usepackage{url}
\usepackage{graphicx}
\usepackage{cite}
\usepackage{marginnote}

\newtheorem{theorem}{Theorem}

\newtheorem{proposition}[theorem]{Proposition}
\newtheorem{remark}[theorem]{Remark}

\newcommand\N{{\mathbb N}}
\newcommand\R{{\mathbb R}}

\newcommand\E{{\mathbb E}}

\newcommand{\Var}{{\rm Var}}

\newcommand{\Span}{{\rm span}}

\begin{document}
\title{Conditions for the finiteness of the moments of the  volume of level sets.}
\author{D. Armentano\thanks{CMAT, 
Universidad de la Rep\'{u}blica, Montevideo, Uruguay. 
E-mail: diego@cmat.edu.uy.}
\qquad
J-M. Aza\"{i}s\thanks{IMT, Universit\'{e} de Toulouse, 
Toulouse, France. Email: jean-marc.azais@math.univ-toulouse.fr}
\qquad
David Ginsbourger \thanks{Idiap Research Institute, Martigny, Switzerland, and Institute of Mathematical Statistics and Actuarial Science, University of Bern, Bern, Switzerland. Emails: ginsbourger@idiap.ch; ginsbourger@stat.unibe.ch.}
\qquad
J. R. Le\'{o}n\thanks{IMERL, 
Universidad de la Rep\'{u}blica, Montevideo, Uruguay.
E-mail: rlramos@fing.edu.uy and Universidad Central de Venezuela. Escuela de Matem\'atica.}}

\maketitle
\makeatother
\abstract{Let   $X(t)$   be a Gaussian random field  $\R^d\to\R$.
Using the notion of $(d-1)$-\emph{integral geometric
measures},
we establish a relation between  (a) the volume  of the level set (b) the number of crossings  of the restriction of the random field to a line. 
Using this relation  we prove the equivalence between the finiteness of the expectation  and the finiteness of the second spectral moment matrix. Sufficient conditions for  finiteness of higher moments are also established.
}

\section{Introduction}

 Let $X(t)$  be a centered, stationary, Gaussian random field
$$X:\Omega\times \R^d\to\R,$$
with continuous sample paths.
By a scaling argument, and without loss of generality, we may assume that  $X(t)$ is centered with variance 1. 
On the other hand, for a given $u\in \R$,  let us consider the level set restricted to some compact set $K \subset \R^d$ 
$$ 
C_{u,K} := \{t \in K\, : X(t)=u\}.
$$
If the sample paths of $X(t)$  are almost surely  (a.s.) differentiable and   if  a. s.
 there exist no point $t$ such that $X(t) =u, \nabla X(t) =0$ 
 (where $\nabla X(t)$ is $X$'s gradient), then by the implicit function theorem, $C_{u,K} $ is almost surely a manifold  and its $(d-1)$-volume is well defined  and coincides with its $(d-1)$-Hausdorff measure, namely, $\mathcal H_{d-1}(C_{u,K})$. Under some non-degeneracy hypothesis, the Kac-Rice formula  (KRF) Aza\"is-Wschebor\cite{AW}) gives an expression  for the moments  of this measure.  If we consider the expectation, the compactness  of the set $K$ and the KRF imply that the first moment is finite, so we already have a sufficient condition of finiteness (but as we will see, the latter is not necessary). For higher moments  the  KRF provides a multiple integral, the integrand of which is degenerate on the diagonal so the study of finiteness is not straightforward. 
 
 When $d=1$,  $C_{u,K}$ is a.s. a set of points and its measure is just the number of points.  We have several result on finiteness of moments, see Sections \ref{section2} and  \ref{section3}. They all use at some stage the Intermediate Value Theorem. Unfortunately these methods are completely  inoperative in  higher dimensions. 
 Here, we appeal to integral geometry in order to establish dimension-independent necessary and sufficient conditions of almost sure finiteness of level set volumes that boil down to one-dimensional results.

 In Section \ref{sec:IGM} we recall the definition of the $(d-1)$-dimensional \emph{integral-geometric measure}, which is defined  as the integral of the number of points over a family of lines.
 
 Our three  main results follow 
 \begin{itemize}
 \item In Section \ref{section2} we establish the equivalence  between (a) the finiteness of the expectation  of the   $(d-1)$-dimensional
 \emph{integral-geometric measure} of the level set and (b) the finiteness of the second spectral moment matrix. This result gives a simpler presentation and shorter proof of the results  of Wschebor  \cite{W} which  uses De Gorgi perimeters.
 
\item In Section \ref{section3} we give sufficient conditions for finiteness of the second moment  (Theorem \ref{C1})  using the Geman condition (See \cite{Ge}).
 
 \item In the same section,  we prove finiteness all moments (Theorem \ref{t:3}), under some conditions, when the sample paths are smooth. 
 \end {itemize}

\section{Integral geometric measure, Crofton formula}\label{sec:IGM}

Let $B$ be a Borel set in $\R^d$. 
Following Morgan\cite{Morgan} (and also Federer \cite{Fe}) we define the $(d-1)$-\emph{integral geometric measure} of $B$
 by
\begin{equation}\label{eq:igm}
  \mathcal{I}_{d-1}(B):=c_{d-1}\int_{v\in S^{d-1}} \left(\int_{y\in v^\perp} \#
  \left\{ B\cap \ell_{v,y} \right\} \, d\mathcal{H}_{d-1}(y)\right)\,d
  S^{d-1}(v)
\end{equation}
where $S^{d-1}$ is the unit sphere in $\R^d$ with its induced Riemannian
measure, and $\ell_{v,y}$ is the affine linear space $\{y+tv:\,
t\in\R\}$. The constant can be easily computed, using the Crofton formula below and considering the particular case of the sphere, yielding, 
$$
c_{d-1}=\frac{\Gamma\left(\frac{d+1}{2}\right)}{2\pi^{(d-1)/2}}.
$$
The integrand  in \eqref{eq:igm} is measurable
 (see for example Morgan\cite[page 13]{Morgan}), and since it is non-negative, the integral is always well defined, finite or infinite.

In particular, if $B$ is $(d-1)$-rectifiable, then Crofton's formula \cite{Morgan} p.  31 yields
\begin{equation}\label{Crofton} 
\mathcal H_{d-1}(B)= \mathcal{I}_{d-1}(B),
\end{equation}
where $\mathcal H_{d-1}$ is the $(d-1)$-Hausdorff measure.


\section{Characterisation for the finiteness of the expected volume of
the level set.}\label{section2}

The spectral measure $ F$  of $X(\cdot)$ is a symmetric  measure with  mass one: it is a probability measure.

Let $\Lambda _2$ be the second spectral moment matrix  defined by
$$
(\Lambda _2)_{ij} := \int_{\R^d} \lambda_i \lambda_j dF(\lambda).
$$
This matrix may be finite or infinite, infinite meaning by convention that at least one entry is infinite.

When $\Lambda _2$ is finite, it is easy to  prove that $X(\cdot)$ is differentiable in quadratic mean. If in addition 
 the sample paths are almost surely differentiable (which is a little
 stronger) and  if  a.s.
there exist no point $t$ such that $X(t) =u, \nabla X(t) =0$, we have:
 
 \begin{itemize}
   \item the level set $C_{u,K}$ is almost surely a submanifold of codimension $1$,
   and its Riemannian volume  can be defined and
     coincides with its $(d-1)$-Hausdorff measure, namely, $\mathcal H_{d-1}(C_{u,K})$; 
\item the KRF (see  Adler-Taylor \cite{at} or Aza\"is-Wschebor\cite{AW}) implies that 
  \begin{eqnarray} \label{rice}
  \nonumber\E(\mathcal H_{d-1}(C_{u,K}))&=&  \E( \mathcal I_{d-1}(C_{u,K}))= \mathcal{L}_d(K)  \E (\|X'(0)\|)  \frac{
      e^{-u^2/2 }}{\sqrt{2\pi}}\\
&=&  \mathcal{L}_d(K) \mathcal F (\Lambda_2)  e^{-u^2/2 },
  \end{eqnarray}
     where $ \mathcal I_{d-1}(C_{u,K})$ is the the $(d-1)$-dimensional
 \emph{integral-geometric measure}
 defined above,  $\mathcal{L}_d$   is the Lebesgue measure on $\R^d$ and
 \begin{equation}\label{e:f}
   \mathcal F (\Lambda_2)  : =   \frac{1} {(2\pi)^{(d+1)/2}} \int_{z\in\R^d}
   (z^\top \Lambda_2 z)^{1/2}
   e^{-\|z\|^2/2} d\mathcal{L}_d(z).
   \end{equation}
The second equality  in \eqref{rice} is the   true Kac-Rice formula, the third  is due to classical integration.
  \end{itemize} 
  
We need to extend  the definition
 of $  \mathcal F (\Lambda_2) $ by setting it to $+\infty$ when
 $\Lambda_2$ is infinite.  
 
  So we consider the following relation: 
 \begin{equation} \label{rice2}
   \E( \mathcal I_{d-1}(C_{u,K}) )=\mathcal{L}_d(K)\mathcal F (\Lambda_2)  e^{-u^2/2 }.
  \end{equation}
 Note that its terms on both hand sides  are now always  well defined, finite or infinite. \bigskip
 
 The goal  of this section  is  to prove that in a broad sense  this formula is always true  : 
   \begin{itemize}
  \item Whenever $\Lambda_2$ is finite, of course the RHS of \eqref{rice} is finite, but  also the LHS is finite also and equality holds true.
   \item  If $\Lambda_2$ is infinite then both sides  of \eqref{rice} are infinite. 
  \end{itemize}
  
   Such a kind of property   is known since the work  of Cram\'er-Leadbetter \cite{CL} for $d=1$ and from the  work of Wschebor \cite{W} for $d>1$. Our proof  uses  Cramer-Leadbetter's result and  generalised Crofton's formula.

\noindent We first  recall   a result due to Cram\'er-Leadbetter, main result of Section 10.3 of \cite{CL}.\\
The expected number of crossings $N_u([0,T]$ of a stationary processes with any level $u$  on an interval $[0,T]$ is finite if and only if $\lambda_2<\infty$, where $\lambda_2$ denotes the second spectral moment.\\
In case $\lambda_2$ is finite we have furthermore 
$$
\E(N_u([0,T]) =  \frac {T }{\pi} \sqrt{\lambda_2} e^{-u^2/2}.
$$

This result is based on polygonal approximation  and intermediate values theorem, so it heavily  relies  on one-dimensional settings.

 We  now turn  to our first  main result.

\begin{theorem}\label{theo1}
Let $X(t)$ be a centered, stationary random field $X:\R^d\to\R,$
with continuous sample paths. Then, we have equivalence between:
\begin{itemize}
  \item $\E(\mathcal I_{d-1}(C_{u,K}))<\infty$,
\item  The second spectral moment  matrix  $\Lambda_2$ is finite.
\end{itemize}
In such a case we have 
$$
  \E(\mathcal I_{d-1}(C_{u,K})) = \mathcal{L}_d(K)\mathcal F (\Lambda_2)  e^{-u^2/2 }
$$

\end{theorem}

\begin{proof}
  Since $X$ is almost surely continuous, then $C_{u,K}$ is a Borel set on
  $\R^d$ a.s., and therefore its integral geometric measure is well
  defined.
  By Fubini theorem  we get that
$$
  \E(\mathcal{I}_{d-1}(C_{u,K}) )=c_{d-1}\int_{v\in S^{d-1}} \left(\int_{y\in
  v^\perp} 
  \E( \#   \left\{ C_{u,K}\cap \ell_{v,y} \right\}) 
  \, d\mathcal{H}_{d-1}(y)\right)\,dS^{d-1}(v)
  $$
As a matter of fact, because of stationarity of the process and by
  Cram\'er-Leadbetter applied to the process $t\mapsto X(y+tv)$,  
 we get
 $$
  \E( \#   \left\{ C_{u,K}\cap \ell_{v,y} \right\})=
  \mathcal{H}_1(K\cap \ell_{v,y})
  \sqrt{v^\top \Lambda_2 v} \, \frac{1}{\pi}e^{-u^2/2}. 
$$
Then, $\E(\mathcal{I}_{d-1}(C_{u,K}))$ is equal to
 \begin{equation*}
      e^{-u^2/2}\frac{c_{d-1}}{\pi}\cdot 
    \int_{v\in S^{d-1}}  \sqrt{v^\top \Lambda_2 v}\,
    \left( 
  \int_{y\in v^\perp}\mathcal{H}_1(K\cap
  \ell_{v,y})\,d\mathcal{H}_{d-1}(y)\right)\,dS^{d-1}(v)
 \end{equation*} 
and by Fubini, we obtain that 
\begin{equation}\label{e:1}
  \E(\mathcal{I}_{d-1}(C_{u,K}))
=\mathcal{L}_{d}(K)
e^{-u^2/2}\frac{c_{d-1}}{\pi}     
  \int_{S^{d-1}}  \sqrt{v^\top \Lambda_2 v}\, dS^{d-1}(v).
    \end{equation}

  Integrating in  polar coordinates the expression
  $\mathcal{F}(\Lambda_2)$, given in (\ref{e:f}), we obtain
  \begin{align*}
    \mathcal{F}(\Lambda_2) &= 
    \frac{1} {(2\pi)^{(d+1)/2}} \int_{0}^{+\infty}\rho^d
    e^{-\rho^2/2}\,d\rho\,\int_{v\in S^{d-1}} 
    (v^\top \Lambda_2 v)^{1/2} dS^{d-1}(v).
  \end{align*}
Furthermore, making the change of variable $u=\rho^2/2$, it is
  straighforward to conclude that
  \begin{equation}\label{eq:flambda}
    \mathcal{F}(\Lambda_2) =\frac{c_{d-1}}{\pi}\int_{S^{d-1}}
    \sqrt{v^\top \Lambda_2 v}\, dS^{d-1}(v), 
\end{equation}
  and therefore from (\ref{e:1}) yields
  \begin{equation}\label{eq:main}
   \E(\mathcal{I}_{d-1}(C_{u,K}))
=\mathcal{L}_{d}(K)
e^{-u^2/2} \mathcal{F}(\Lambda_2).
\end{equation}

We consider the two following cases.

\begin{itemize}
\item  When $\Lambda_2$ is finite  the integral  on the RHS of
\eqref{eq:main} is finite and therefore we get the desired result in this
    case.
 \item  
 When $\Lambda_2$ is infinite,  this means that   this matrix has at least  one infinite element. In such a case we define the  linear subspace 
$$
G(\Lambda_2) := \{ v \in 
\R^{d} : v^\top \Lambda_2 v < +\infty\}.
$$
We prove  that $G(\Lambda_2) $ is    of dimension strictly
  smaller than $d$.
Let   $ v_1 , \ldots, v_{d_0}$ be a maximal set of linearly independent vectors of  $G(\Lambda_2)$.
    Then by standard linear algebra:
\begin{itemize}
\item the space $\Span(  v_1 , \ldots, v_{d_0})$ generated by $ v_1 ,
  \ldots, v_{d_0}$ is in  $G(\Lambda_2) $.  This implies that $d_0<d$,
 \item  for every $v \notin \Span(  v_1 , \ldots, v_{d_0})$: $  v^\top
   \Lambda_2 v =+ \infty$ (unless $ v_1 , \ldots, v_{d_0}$ is not maximal) ,
 \item this implies that   $G(\Lambda_2)  =  \Span(  v_1 , \ldots, v_{d_0})$. 
 \end{itemize}
In conclusion the integrand  in \eqref{eq:main} is almost everywhere infinite so the integral is infinite and  by consequence  the expectation of the integral geometric measure 
  is infinite.
  \end{itemize}

\end{proof}
\section{Finitness of $k$-moments of the volume of the level set}\label{section3}

Using Formula (\ref{eq:igm}) it is possible to obtain sufficient conditions under which  the random variable $\mathcal{I}_{d-1}(B)$ has finite moments. To illustratethis we will first consider the second moment. Thus we have the following

\begin{theorem} \label{C1}Let assume that
\begin{itemize}
\item The second spectral moment  matrix  $\Lambda_2$ is non-degenerate. 
 \item There exists $\delta >0$ such that the spectral mesure $F$ satisfies
$$\int_{\R^d}||\lambda ||^{2+\delta}dF(\lambda)<\infty.$$
\end{itemize}
Then we have
$$\displaystyle
  \E(\mathcal{I}_{d-1}(C_{u,K}) )^2<\infty.$$
\end{theorem}

%
\begin{remark}
Let us point out that under the assumption that $\int_{\R^d}||\lambda||^{2+\delta}dF(\lambda)$ is finite,  the Kolmogorov-Chentsov criterion implies that the field $X$ has a.s. $ \mathcal C^1$ sample paths. 
Thus the following equality takes place \\ $\displaystyle
  \E(\mathcal{I}_{d-1}(C_{u,K}) )^2= \E(\mathcal H_{d-1}(C_{u,K}))^2.$ Moreover, the Riemannian volume  of $C_{u,K}$ can be defined and coincides with its $(d-1)$-Hausdorff measure.
\end{remark}
\begin{proof} Without loss of generality we assume that $ \delta <2$. Let  $r$ be the covariance function of $X$, and let us first consider the field restricted to the line  $\ell_{y,v}: y+tv, t\in \R$: $\tilde X_{y,v}(t)=X(y+tv)$. Its covariance function is given by
$$\ r_{v}(t)=\E[X(y+tv)X(y)]=r(tv).$$ 
Note that because of stationarity  it does not depend on $y$.

It is sufficient to prove the assertion of the theorem for  a set $K$  being a  centred ball  $B_a$ with sufficiently small diameter $a$.  In that case  note that the integral in the right-hand side of  \eqref{eq:igm} is finite since for $ |y |>a$ the integrand vanishes. 
Since the  second spectral moment  matrix is finite and non degenerate 
$$ \int_{\R^d}  \langle \lambda, v \rangle ^2 dF(\lambda),
$$
is bounded below and above.  On the other hand, using a monotone convergence argument, as $b$ tends to infinity 

\begin{equation} \label{jma:e0}
\int_{\R^d \setminus B_b}  \langle \lambda, v \rangle ^2 dF(\lambda)  \leq  \int_{\R^d \setminus B_b} || \lambda || ^2 dF(\lambda) \to 0.
\end{equation}
So it is easy to conclude that  for $b$ sufficiently large,  for any $v  \in S^{d-1}$
\begin{equation} \label{jma:e1}
\int_{B_b}  \langle \lambda, v \rangle ^2 dF(\lambda) >1/2 \int_{\R^d}  \langle \lambda, v \rangle ^2 dF(\lambda).
\end{equation}

In the rest of the paper $\mathbf C$ will denote  some  unimportant constant, its value may change from an occurence to another.

Applying the Jensen inequality (with respect to the integral) yields 
\begin{multline*}
  \E(\mathcal{I}_{d-1}(C_{u,K}) )^2
  \\
\le\mathbf C\,c^2_{d-1}\int_{v\in S^{d-1}}\int_{y\in
  v^\perp} 
  \E( \#   \{ C_{u,K}\cap \ell_{y,v} \})^2
  \, d\mathcal{H}_{d-1}(y)\,dS^{d-1}(v).
 \end{multline*}
  As already remarked,   the integral is over a bounded domain and 
 it is sufficient to prove  that  the integrand  is uniformly bounded. 
 
 Remark also that $  B_b\cap \ell_{y,v} $ is always a centred interval with length $2c$ less that $2a$.  Consequently,
 $$
 \E( \#   \{ C_{u,K}\cap \ell_{y,v} \})^2 \leq \E( \#   \{ C_{u,K}\cap \ell_{0,v} \})^2.
$$
It remains to prove that 
 $$
  \E( \#   \{ C_{u,K}\cap \ell_{0,v} \})^2
  $$ is uniformly bounded.
  
  In fact, because of the Rolle theorem, if  $ U_u$ is the number of up-crossings of the level $u$  on the line $\ell_{0,v}$, then
  $$
  \#   \{ C_{u,K}\cap \ell_{0,v} \} \leq 2U_u +1.
  $$
  So it is sufficient  to bound the second moment of $U_u$ and even, because we have proved in the previous section that the first moment is uniformly bounded, it is sufficient to bound the second  factorial moment. Since the variance has been assumed to be 1 and $\Lambda_2$ is non-degenerate, the Kac-Rice formula applies and yields 
  \begin{align*}
  & \E(U_u(U_u-1)) 
  \\&=  \int_{-a} ^a \int_{-a} ^a \E \big(X'^+(s) X'^+(t) \big| X(s) =X(t) =u \big) \frac {1}{2 \pi }  \frac {1}{\sqrt{1-r_v^2(s-t)}} ds dt
  \\
  & \leq \mathbf C  \int_0 ^{2a} (2a -\tau ) \E \big(X'^+(0) X'^+(\tau) \big| X(0) =X(\tau) =u \big)  \frac {1}{\sqrt{1-r_v^2(\tau)}}d\tau,
  \end{align*}
   where $X$  stands for $\tilde X_{0,v}$.

   By a standard regression formula, see for example  \cite{AW} page 99, 
   $$
   \E(X'(0) \big|X(0) =X(\tau) =u) = -\E(X'(\tau) \big|X(0) =X(\tau) =u)  = \frac{-r_v'(\tau) u}{1+r_v(\tau)}.
   $$
   Also, 
   \begin{align*}
   \sigma_v^2(\tau) :&= \Var(X'(0) | X(0) = X(\tau) =u ) \\
   &=\Var(X'(\tau)  | X(0) = X(\tau) =u )= \frac{ \lambda_{2,v}(1-r_v(\tau)) - r_v'^2(\tau)}{1-r_v^2(\tau)}.
   \end{align*}
  Set $\theta_v(\tau )  := r_v(\tau)-1+\lambda_{2,v} \tau^2/2$
  using the inequality $ z^+t^+ \leq (z+t)^2/4$  and  the fact that $ \theta_v(\tau )$, $ \theta_v'(\tau )$, $ \theta_v''(\tau )$ are non-negative  we get 
  $$
   \E(U_u(U_u-1))  \leq \mathbf C a \int _0 ^{2a}  2 \lambda_{2,v} \tau   \theta_v'(\tau)\big(1-r_v^2(\tau) \big)^{-3/2}.
  $$
 Now, there exists a constant $  \mathbf {C_0 }$  such that
 $$ 
 0<w<1 \mbox{ implies  that }  1- \cos (w)  \geq  \mathbf {C_0 } w^2. 
$$ 
This implies in turn that  for $\tau <1/b$ where $b$ has been defined in  \eqref{jma:e1}
   \begin{align*}
1-r_v(\tau)  &= \int _0 ^{+\infty}  1-\cos(\lambda \tau) \ dF_v(\lambda) \\
&\geq  \int _0 ^{1/\tau} 1-\cos(\lambda \tau) dF_v(\lambda) \\
&\geq \mathbf {C_0 }  \tau^2  \int _0 ^{1/\tau} \lambda^2  dF_v(\lambda) \geq \frac 1 2 
 \mathbf {C_0 }  \tau^2  \int _0 ^{+ \infty} \lambda^2  dF_v(\lambda)
 \geq \mathbf {C } \tau^2,
    \end{align*}
 where $F_v$ is the spectral measure  along the line $ \ell_{0,v}$ (for convenience it is on $(0,+\infty))$. The penultimate equality uses \eqref{jma:e1}, the last inequality  is due to the fact that 
 $\Lambda_2$ is non-degenerate.
 \medskip
 
On the other hand  it is direct to prove that $1-r_v(\tau) \leq \lambda_{2,v}\tau^2$  and, by compactness, the quantity $\lambda_{2,v}$  is bounded as a function of $v$ giving that $1 + r_v(\tau) \geq 1$ as soon as the radius $a$ of the ball is sufficiently small.  This yields 
$$
1-r_v^2(\tau)  \geq \mathbf {C } \tau^2 .$$
%
As a consequence
\begin{equation}\label{e:geman}
   \E(U_u(U_u-1))  \leq \mathbf C \lambda_{2,v} a \int _0 ^{2a}  2   \frac{\theta_v'(\tau)}{\tau^2} d\tau.
   \end{equation}
 The integrand in \eqref{e:geman} can be bounded because
 $$
  \int_0 ^{\infty} \lambda ^{2+\delta}dF_v(\lambda)   \leq  I(\delta) := \int_{\R^d}||\lambda ||^{2+\delta}dF(\lambda) <\infty.
 $$
 We have 
 $$
 \frac{\theta_v'(\tau)}{\tau^2}  = \tau^{-2} \int_0 ^{\infty}  (\tau \lambda_2 - \lambda \sin(\lambda \tau) )dF_v(\lambda)  .
 $$
 Define 
 $$
 R(u):= (u-\sin(u)).
 $$
 Its behaviour  at zero and at infinity implies that for every $\delta$, $ 0<\delta<2$, there exist a constant $\mathcal C_\delta$  such that 
 $$
 0 \leq  R(u) \leq \mathcal C_\delta u^{1+\delta}.
 $$
 This implies that 
 \begin{align*}
  \Big|  \frac{\theta'_v(\tau)}   {\tau^2} \Big|  
  & \leq \tau^{-2} \int_0 ^{\infty}   \lambda | R( \lambda \tau) |  dF_v(\lambda)  \\
 &    \leq  \mathcal C_\delta    \tau^{-2} \int_0 ^{\infty}   \lambda  (\lambda \tau)^{1+\delta} dF_v(\lambda)\\
 & \leq  \mathcal C_\delta    \tau^{\delta-1} 
 \int_0 ^{\infty}   \lambda^{2+\delta} dF_v(\lambda),
 \end{align*}
 Implying the convergence of the integral  in \eqref{e:geman}, uniformely in $v$. 

\end{proof}


Next, we consider a Gaussian field having $C^{\infty}$ sample paths. This is for instance the case of Gaussian random trigonometric polynomials in several variables or  the random plane wave model \cite{Z}. A result of Nualart \& Wschebor, quoted as Theorem 3.6 in the book  \cite{AW}, can be used for obtaining that all the moments of the random variable $\mathcal{I}_{d-1}(C_{u,K})$ are finite. The background result is the following:

\begin{proposition} \label{NW}  Consider a Gaussian process $\chi , \R \to \R$ satisfying  $\Var(\chi(t))>\kappa$ for all $t\in I$ a compact interval of $\R$ and some $\kappa >0$. Then for all $u\in\R$, and $m,p\in\N$ such that $p>2m $, it holds 
  
\begin{equation}
\E[(N_u)^m] \leq C_{p,m}\big[ 1+C+\E\big( \| X^{(p+1)}\|
_{\infty }\big) \big]   \label{dnmw}
\end{equation}
where  $N_u $ is the number of points $t \in I$ such that $\chi(t)  =u$, $C_{p,m}$ is a constant depending only on $p,m$ and the length of the interval $I$, and $C$ is a bound for the density of $\chi(t)$.
\end{proposition}

Let us  assume now that the field $X$ has $C^\infty$ sample paths  and we assume that the variance is bounded  below. As in the proof on Corollary \ref{C1} the process $\tilde X_{y,v}(t)=X(y+tv)$ is a real process but now with $C^\infty$ trajectories. Chose $p = 2m +1$ then from Proposition~\ref{NW}, for every $m$,  
$$ 
\E( \#   \left\{ C_{u,K}\cap \ell_{v,y} \right\})^m<  C_{p,m}\big[ 1+C+\E\big( \| X^{(2m+2)}_{y,v}\|_{\infty }\big) \big] .$$
It is an easy consequence of the Borel-Sudakov-Tsirelson  inequality that $\E\big( \| X^{(2m+2)}_{y,v}\|_{\infty }\big) $ is finite. An argument of  continuity shows that  it is uniformly bounded. A further application of Jensen's inequality gives our third main result 

\begin{theorem} \label{t:3}
Let $ X(t)$ a Gaussian random field  $\R^d   \to \R$   with $C^{\infty}$ sample paths and with variance bounded below. Then for every integer $m$  and every compact set $K$,
$$ \E(\mathcal{I}_{d-1}(C_{u,K}) )^m<\infty.$$
\end{theorem}


\noindent
\textbf{Acknowledgements:}
This work has been partially supported by the French National
Research Agency (ANR) through project PEPITO
(no ANR-14-CE23-0011).  \\
David Ginsbourger would like to thank Andrew Stuart for a  stimulating discussion in Cambridge that has incented the present collaboration, and in turn to thank the Isaac Newton Institute for Mathematical Sciences, Cambridge, for support and hospitality during the programme ``Uncertainty quantification for complex systems: theory and methodologies'' (supported by EPSRC grant no EP/K032208/1).  

\end{document}